\documentclass[
 aip,
 sd,%
 amsmath,amssymb,
preprint,%
nofootinbib]
{revtex4-1}
\usepackage{graphicx}
\usepackage{dcolumn}
\usepackage{bm}
\usepackage{amsthm}
\newtheorem{theorem}{Theorem}

\begin{document}
\title[Macdonald's Theorem]{Macdonald's Theorem for Analytic Functions}

\author{ R.C. McPhedran,\\
 School of Physics, University of Sydney,\\
Sydney, NSW Australia 2006.}
  
\begin{abstract} 
A proof is reconstructed  for a useful theorem on the zeros of derivatives of analytic functions due to H. M. Macdonald, which appears to be now little known. The Theorem states that, if a function $f(z)$ is analytic inside a bounded region bounded by a contour on which the modulus of $f(z)$ is constant, then the number of zeros (counted according to multiplicity) of $f(z)$ and of its derivative in the region differ by unity.
\end{abstract}
\maketitle
\section{Introduction}

The well-known mathematician E. T. Whittaker wrote the Obituary of H.M. Macdonald, published in Nature in 1935\cite{whittaker}.
He described  Macdonald's life and contributions to physical topics, but his comment on a mathematical contribution is of paramount interest here:
\begin{quote}
In the paper on the zeros of Bessel functions[\cite{macdonald}], he gave the result since known as Macdonald's Theorem, that the number of zeros of a function $f(z)$ in the region bounded by a contour at each point of which $|f(z)|=$a constant, exceeds the number of zeros of the derived function $f'(z)$ in the same region by unity, the function $f(z)$ being supposed analytic in the region.
\end{quote}
In the form described by Whittaker, one would imagine that Macdonald's Theorem would be well cited, and would have found a range of uses in numerics and analytics. However, this is not the case, and the Theorem appears to be no longer at all well known.

One reason for this may be that, if one reads the source paper\cite{macdonald}, one does not find the Theorem stated in the generality of Whittaker's enunciation. In fact, Macdonald's treatment is  not at all general, but is couched entirely in terms of the eponymous function $K_n(z)$ and its derivative. Hence, there is an interest in reconstructing the Theorem in the generality of Whittaker's description, in the hope that it may find wider contemporary interest and application. This reconstruction will be found in the next section, and numerical examples with program details will be given in Section 3, which may be useful in pedagogy.
\section{Reconstructed Proof of Macdonald's Theorem}
\begin{theorem}[Macdonald, 1898; Whittaker, 1935] The number of zeros $N_z$ of a non-constant function $f(z)$ in the  bounded region inside a contour at each point of which $|f(z)|=C$, $C$ being a constant, exceeds the number of zeros of the derived function $f'(z)$ in the same region by unity, the function $f(z)$ being supposed analytic in the region. Zeros are counted according to their multiplicity.
\end{theorem}
\begin{proof}
The proof proceeds by induction, and concentrates on the morphology of the lines of constant modulus of $f(z)$ within the region.
We commence with the situation where all zeros are simple.

Consider first the case $N_z=1$, where there is exactly one zero in the region, say at $z_1$. Draw curves of constant modulus surrounding $z_1$, for which the modulus increases towards $C$. Then each point in the region lies on one of these curves, and only one. Indeed, if two  curves for a modulus $C_*<C$ touched inside the region, one centred on $z_1$ and the other on a zero exterior to the region, this would give a contradiction, since  equimodular lines for values intermediate between $C_*$ and $C$ would enclose both zeros, and thus leave the region. Also, a contour of constant modulus  containing no zero (and no pole) could be constructed from those for the exterior zero and those for $z_1$, in contradiction with the Maximum/Minimum Modulus Theorems.  There are thus no zeros of $f'(z)$ inside the region, as required.

Consider next the case $N_z=2$, for which there are zeros $z_1$ and $z_2$ lying inside the region. Construct equimodular lines centred on $z_1$ and $z_2$. As  the specified modulus increases, there must be a value $C_2$ at which the curves first touch, at some point intermediate between $z_1$ and $z_2$, and called by Macdonald a double point. This double point has $f'(z)=0$. Equimodular lines for modulii intermediate between $C_2$ and $C$ enclose both zeros $z_1$ and $z_2$,  and can intersect no other equimodular lines in the region, by the argument given for $N_z=1$.

Consider next the general case $N_z=m>2$. We know that any system of zeros of order $m_s\ge 1$ smaller than $m$ lying inside the contour $|f(z)|=C$ must have $m_s-1$ zeros of $f'(z)$ lying inside the region, and enclosed by equimodular lines surrounding all the $m_s$ zeros. We then take $m_s=m-1$ and consider the equimodular contours around $z_m$ and the system of $m-1$ zeros.
By varying the modulus value, we will arrive at a value for which the two systems of contours touch at a double point, and thereafter enclose all $m$ zeros until $C$ is reached. We thus have added one zero of $f'(z)$ to the $m-2$ which existed within the system of $m-1$ zeros, giving $m-1$ zeros of $f'(z)$ in all (again relying on the argument given for $N_z=1$). 

Next, consider the case  where some or all zeros have multiple order. We deal with this by remarking that a zero of order $N_m$
may be treated as an arbitrary  cluster of  $N_m$ simple zeros separated by distances much smaller than their distances to all other zeros, as far as the behaviour of lines of constant modulus external to the region of this localised cluster is concerned. The localised cluster has within it $N_m-1$ derivative zeros, just as the zero of multiple order has a derivative zero of order $N_m-1$. The proof then proceeds as in the case where all zeros are simple.

\end{proof}

\section{Numerical Examples}

Numerical examples illustrating the argument of Section 2 may be constructed easily in Mathematica, or other systems combining graphics and analytics. The Mathematica code given below constructs Nzin complex numbers, by choosing random real modulii and arguments for each, and uses these to construct the function giving the value of $f(z)$ at $z=x+i y$.

\begin{quote}
Clear[ztabsxy]; 
ztabsxy[Nzin\_] := \\
 Module[\{tabinz, l, tabout, taboutxy\}, \\
  tabinz = RandomReal[\{0, 1\}, Nzin]; \\
  tabout = Table[\{tabinz[[l]], RandomReal[\{0, 2*Pi\}]\}, \{l, 1, Nzin\}];\\ 
  taboutxy = 
   Table[\{tabout[[l, 1]]*Cos[tabout[[l, 2]]], 
     tabout[[l, 1]]*Sin[tabout[[l, 2]]]\}, \{l, 1, 
     Nzin\}]]; \\
     
     Clear[fnval]; \\
fnval[x\_, y\_, Nzin\_, stabxy\_] := 
 Product[(x + I*y - (stabxy[[l, 1]] + I*stabxy[[l, 2]])), \{l, 1, 
   Nzin\}];
\end{quote}

The text at the top of Fig.\ref{fig2z} gives Mathematica commands for constructing contour plots for the case of two zeros. The contours comprise a set generated by a list, and three others which are chosen interactively. This should be done with a given set of complex points, so that the interactive contours can be adjusted for a given case. The aim is to choose the interactive contours to
illustrate the behaviour near double points. For the case shown, the double point corresponds to a logarithmic modulus between -2.39 and -2.415.

\begin{figure}[htb]
\includegraphics[width=6 in, angle=-90]{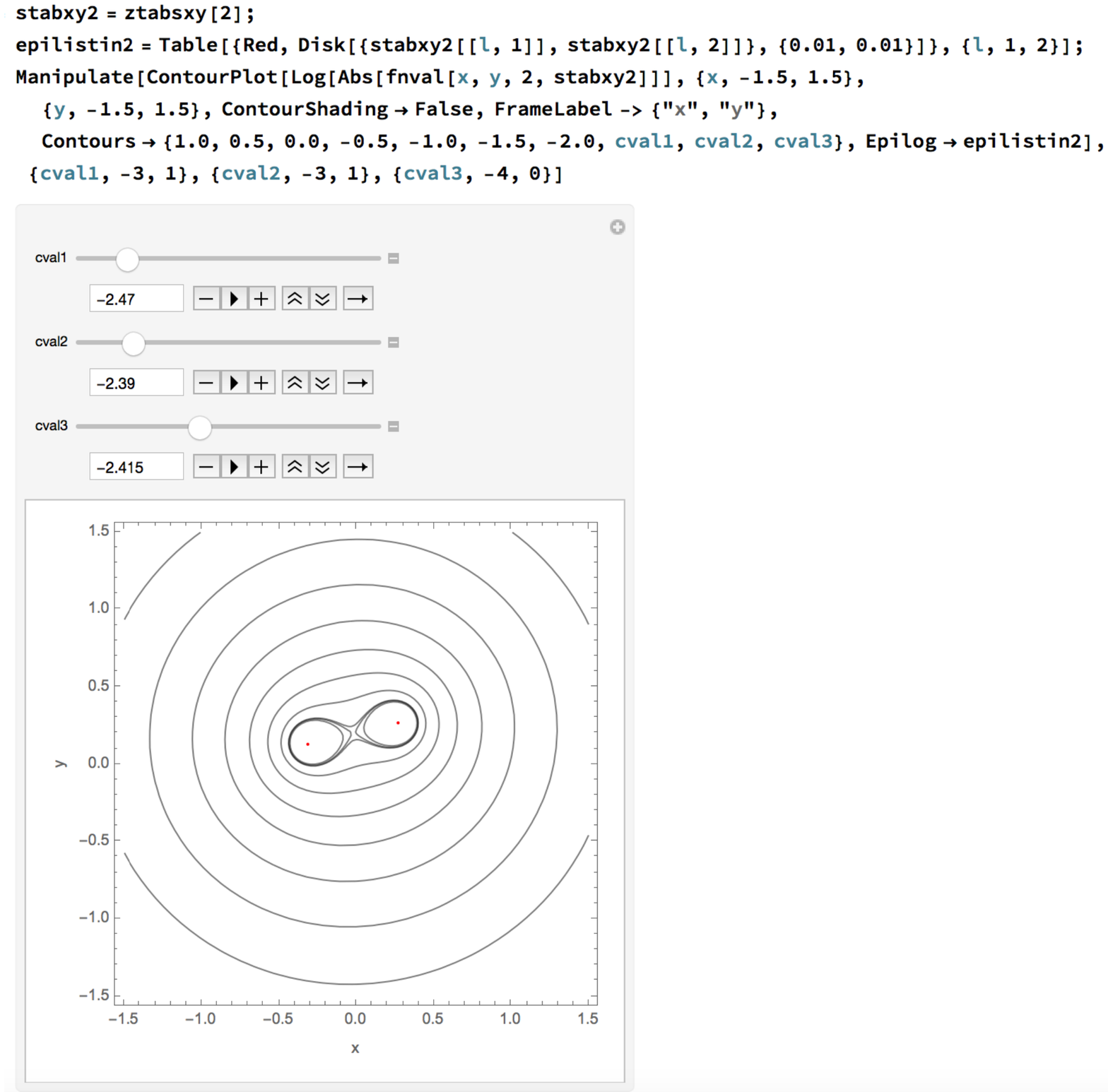}
\caption{Mathematica commands using Manipulate  and example with two zeros randomly chosen with modulus less than unity.The positions of the zeros are indicated by red dots.}
\label{fig2z}
\end{figure}

Examples are also given in Figs. \ref{fig3z}-\ref{fig5z} of three, four and five zeros randomly chosen. As the number of random zeros
increases, the function modulii of interest tend to decrease, which should be reflected in the range of fixed contours and the adjustment range of the interactive contours. The number of interactive contours also naturally needs to increase.
\begin{figure}[htb]
\includegraphics[width=6 in, angle=0]{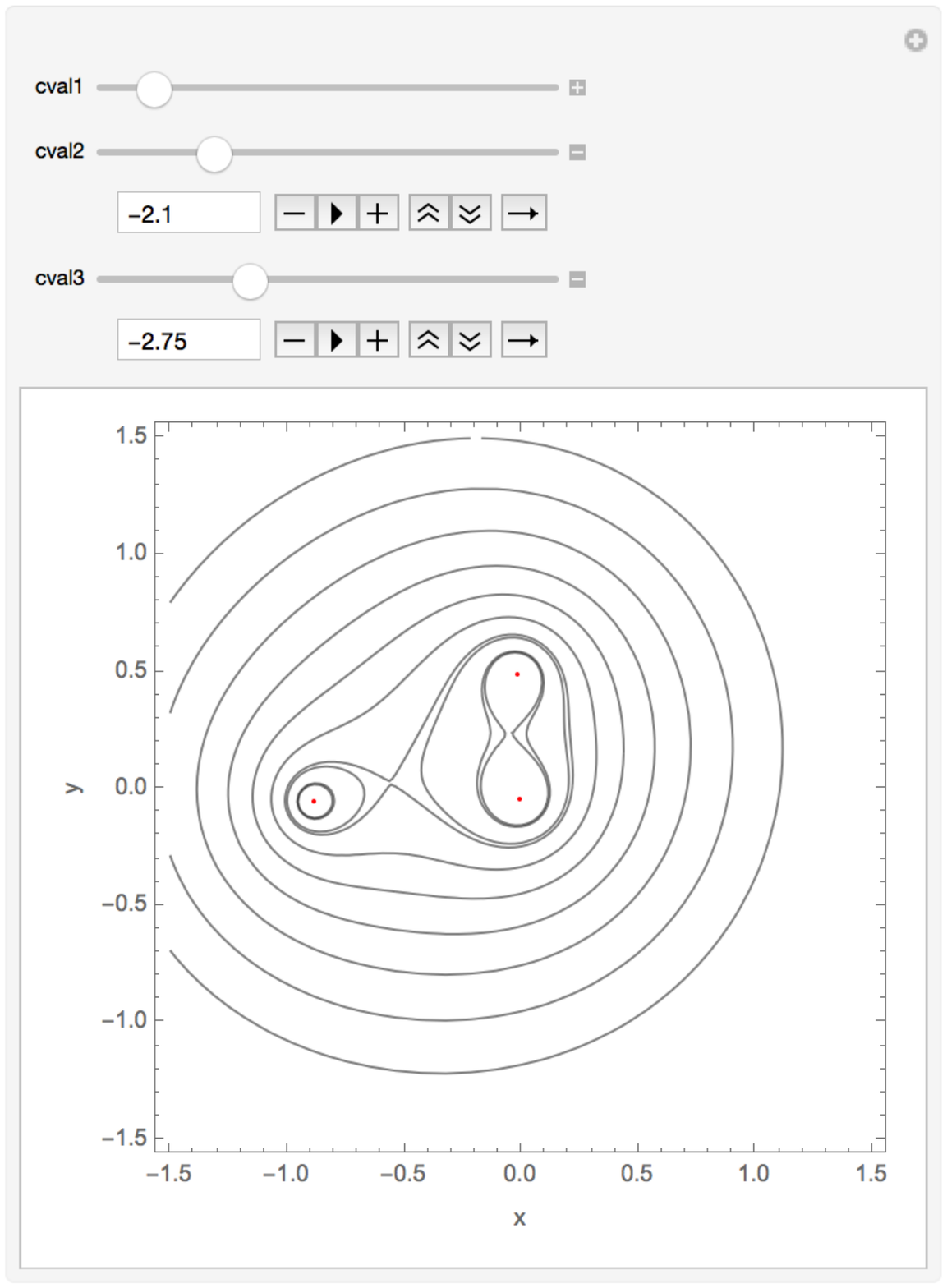}
\caption{Example with three zeros randomly chosen with modulus less than unity.}
\label{fig3z}
\end{figure}

\begin{figure}[htb]
\includegraphics[width=6 in, angle=0]{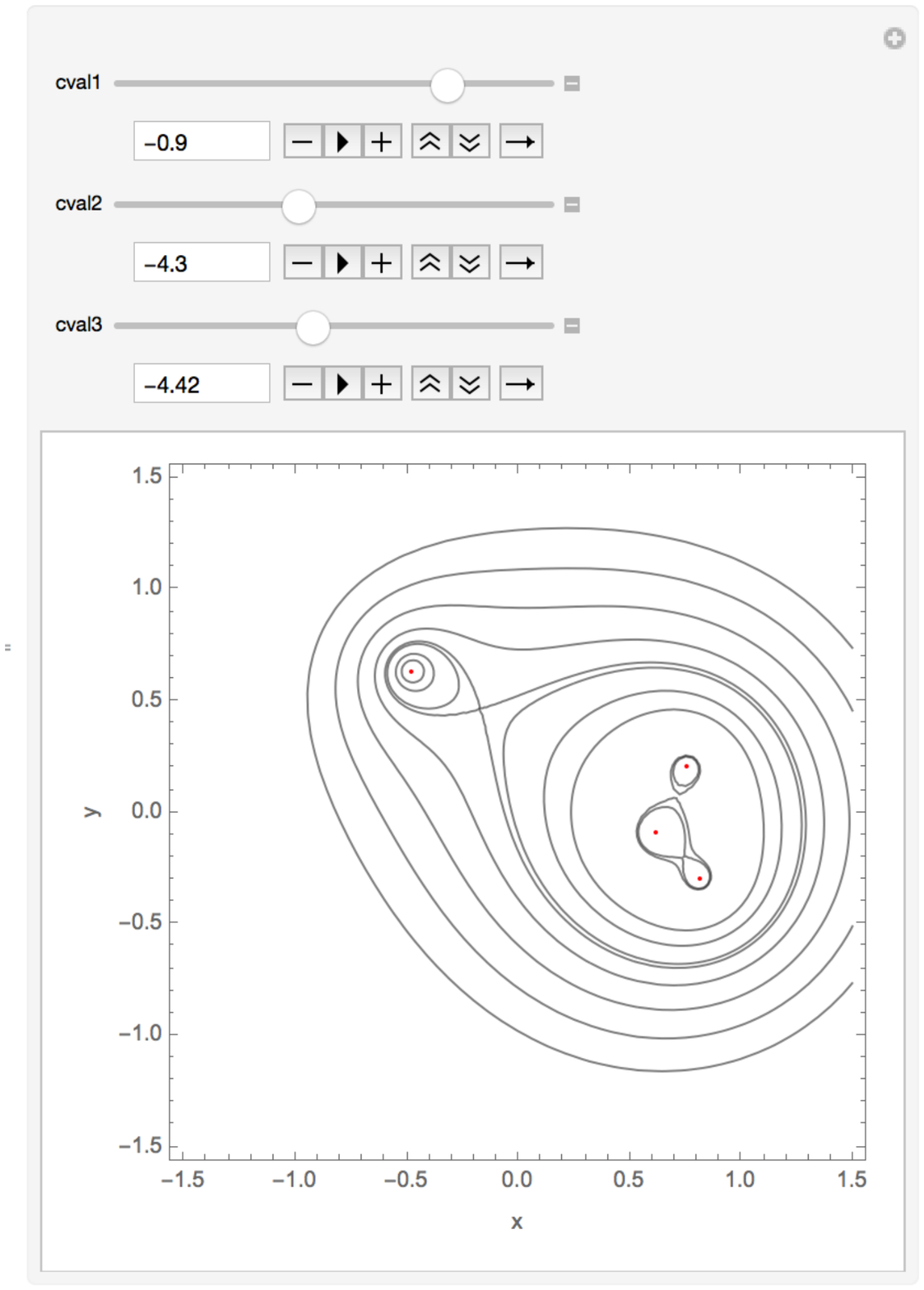}
\caption{Example with four zeros randomly chosen with modulus less than unity.}
\label{fig4z}
\end{figure}

\begin{figure}[htb]
\includegraphics[width=6 in, angle=0]{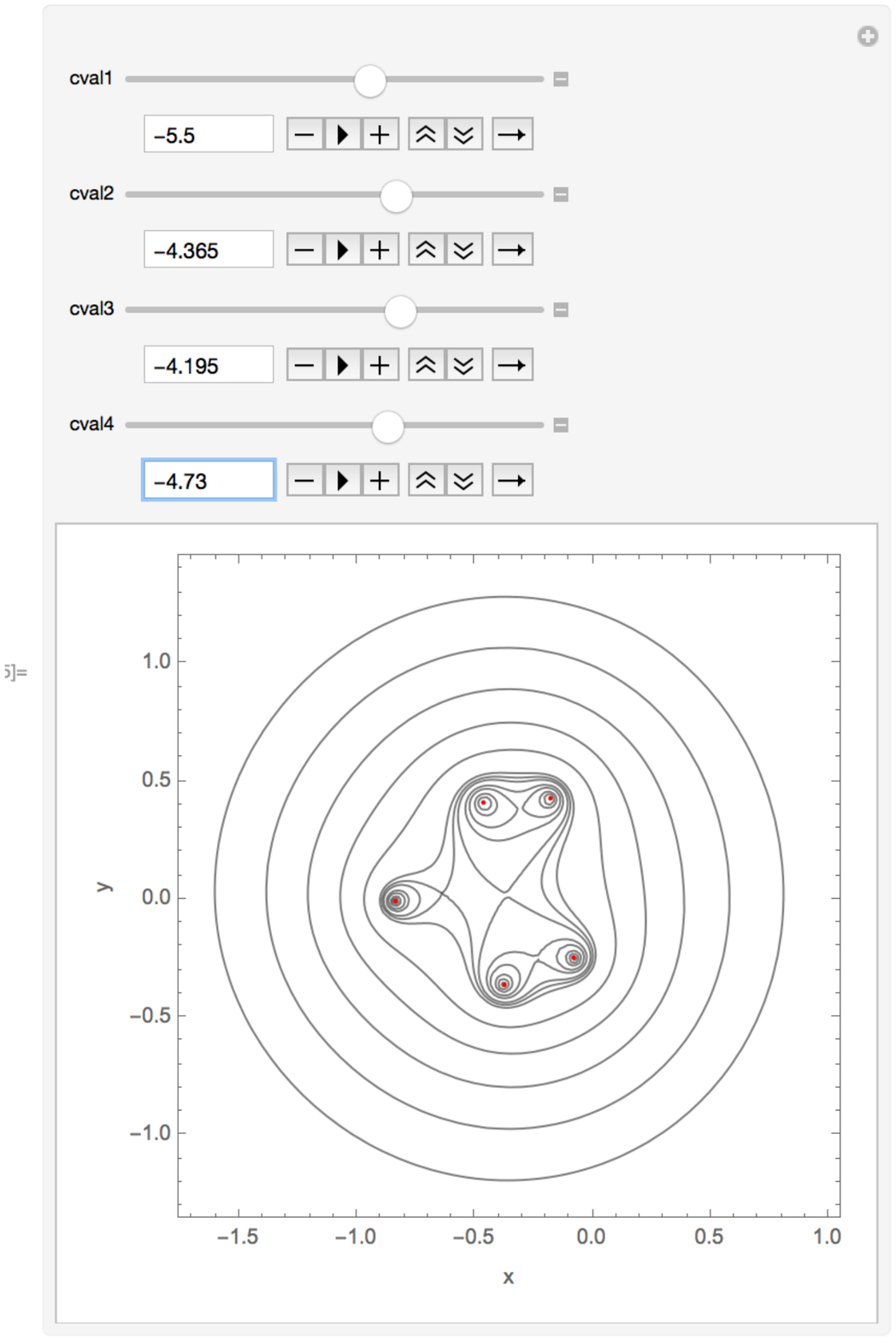}
\caption{Example with five zeros randomly chosen with modulus less than unity.}
\label{fig5z}
\end{figure}

It is interesting that the examples in Figs. \ref{fig3z}-\ref{fig5z} show a clear break up into the clusters of $m-1$ zeros  and 1 zero 
utilised in the proof of Section 2.

\end{document}